\documentclass[a4paper,12pt,reqno]{amsart}
\usepackage{eurosym}
\usepackage{amsfonts}
\usepackage{amsmath,amsthm,amssymb}
\usepackage{graphicx, color}
\usepackage{cite}
\usepackage{float}
\nonstopmode \numberwithin{equation}{section}

\newtheorem{theorem}{Theorem}
\newtheorem{corollary}{Corollary}[section]

\allowdisplaybreaks

\begin{document}
\title[Dynamic $\mathtt{K}$-Struve Sumudu Solutions for Fractional Kinetic Equations]{Dynamic $\mathtt{K}$-Struve Sumudu Solutions for Fractional Kinetic Equations}

\author{K.S. Nisar, F.B.M. Belgacem}

\address{K. S. Nisar : Department of Mathematics, College of Arts and
Science at Wadi Aldawaser, Prince Sattam bin Abdulaziz University, Riyadh
region 11991, Saudi Arabia}
\email{ksnisar1@gmail.com, n.sooppy@psau.edu.sa }

\address{F.B.M. Belgacem: Department of Mathematics,  Faculty of Basic Education, PAAET, 
Al-Ardhiya, Kuwait}
\email{fbmbelgacem@gmail.com}


\begin{abstract}
 In this present study, we investigate solutions for fractional kinetic equations, involving $\mathtt{k}$-Struve functions using Sumudu transform.  The methodology and results can be considered and applied to various related fractional problems in mathematical physics.
\end{abstract}
\maketitle

\section{Introduction}
The Struve function $H_{p}\left(x\right)$ introduced by Hermann Struve in 1882, defined for $p\in \mathbb{C}$ by
\begin{equation}
H_{p}\left( x\right) :=\sum_{k=0}^{\infty }\frac{\left( -1\right) ^{k}}{%
\Gamma \left( k+3/2\right) \Gamma \left( k+p+\frac{3}{2}\right) }\left( 
\tfrac{x}{2}\right) ^{2k+p+1},  \label{Struve-1}
\end{equation}

is the particular solutions of the non-homogeneous Bessel differential equations, given by, 
\begin{equation}
x^{2}y^{^{\prime \prime }}\left( x\right) +xy^{^{\prime }}\left( x\right)
+\left( x^{2}-p^{2}\right) y\left( x\right) =\frac{4\left( \frac{x}{2}%
\right) ^{p+1}}{\sqrt{\pi }\Gamma \left( p+1/2\right) }.  \label{bde-1}
\end{equation}
The homogeneous version of \eqref{bde-1} have Bessel functions of the first kind, denoted as $J_{p}(x)$, for solutions, which are finite at $x = 0$, when $p$ a positive fraction and all integers \cite{Belgacem2010}, while tend diverge for negative fractions,$p$. The Struve functions occur in certain areas of physics and applied mathematics, for example, in water-wave and surface-wave problems \cite{Ahmadi,Hirata}, as well as in problems on unsteady aerodynamics \cite{Shaw}. The Struve functions are also important in particle quantum dynamical studies of spin
decoherence \cite{Shao} and nanotubes \cite{Pedersen}. For more details about Struve functions, their generalizations and properties, the esteemed reader is invited to consider references, \cite{Orhan,Yagmur, Bhowmick, Bhowmick2, Kant, Singh1, Singh2, Singh3, Singh4, Singh5}.
Recently, Nisar et al. \cite{Nisar-Saiful} introduced and studied various properties of $\mathtt{k}$-Struve function $\mathtt{S}_{\nu,c}^{\mathtt{k}}$ defined by
\begin{equation}\label{k-Struve}
\mathtt{S}_{\nu,c}^{\mathtt{k}}(x):=\sum_{r=0}^{\infty}\frac{(-c)^r}
{\Gamma_{\mathtt{k}}(r\mathtt{k}+\nu+\frac{3\mathtt{k}}{2})\Gamma(r+\frac{3}{2})}
\left(\frac{x}{2}\right)^{2r+\frac{\nu}{\mathtt{k}}+1}.
\end{equation}
The sumudu transform of $\mathtt{k}-$Struve function is given by
\begin{align*}
S\left[\mathtt{S}_{\nu,c}^{\mathtt{k}}(x)\right]&=\int_{0}^{\infty}e^{-t}\mathtt{S}_{\nu,c}^{\mathtt{k}}(ut)dt\\
&=\int_{0}^{\infty}e^{-t}\sum_{r=0}^{\infty}\frac{(-c)^{r}}{\Gamma_{\mathtt{k}}\left(r\mathtt{k}+\nu+\frac{3\mathtt{k}}{2}\right)\Gamma\left(r+\frac{3}{2}\right)}\left(\frac{ut}{2}\right)^{2r+\frac{\nu}{\mathtt{k}}+1}dt\\
&=\sum_{r=0}^{\infty}\frac{(-c)^{r}}{\Gamma_{\mathtt{k}}\left(r\mathtt{k}+\nu+\frac{3\mathtt{k}}{2}\right)\Gamma\left(r+\frac{3}{2}\right)}\int_{0}^{\infty}e^{-t}\left(\frac{ut}{2}\right)^{\frac{\nu}{\mathtt{k}}+2r}dt\\
&=\sum_{r=0}^{\infty}\frac{(-c)^{r}\Gamma\left(\frac{\nu}{\mathtt{k}}+2r+2\right)}{\Gamma_{\mathtt{k}}\left(r\mathtt{k}+\nu+\frac{3\mathtt{k}}{2}\right)\Gamma\left(r+\frac{3}{2}\right)}\left(\frac{u}{2}\right)^{\frac{\nu}{\mathtt{k}}+1+2r}
\end{align*}
Now, using 
\begin{align}\label{gammak}
\Gamma_{\mathtt{k}}\left(\gamma\right)={\mathtt{k}}^{\frac{\gamma}{\mathtt{k}}-1}\Gamma\left(\frac{\gamma}{\mathtt{k}}\right)
\end{align}
we have the following
\begin{align*}
S\left[\mathtt{S}_{\nu,c}^{\mathtt{k}}(x)\right]=\sum_{r=0}^{\infty}\frac{(-c)^{r}\Gamma\left(\frac{\nu}{\mathtt{k}}+2r+2\right)}{{\mathtt{k}}^{r+\frac{\nu}{\mathtt{k}}+\frac{1}{2}}\Gamma\left(r+\frac{\nu}{\mathtt{k}}+\frac{3}{2}\right)\Gamma\left(r+\frac{3}{2}\right)}\left(\frac{u}{2}\right)^{\frac{\nu}{\mathtt{k}}+1+2r}.
\end{align*}
Denoting the left hand side by $G(u)$, we have
\begin{align}\label{Gu}
G\left(u\right)&=S\left[\mathtt{S}_{\nu,c}^{\mathtt{k}}(t);u\right]\notag\\
&=\left(\frac{u}{2}\right)^{\frac{\nu}{\mathtt{k}}+1}k^{-\frac{1}{2}-\frac{\nu}{\mathtt{k}}}
{}_{2}\Psi _{2}\left[
\begin{array}{c}
(\frac{\nu}{\mathtt{k}}+2,2),(1,1) \\
(\frac{\nu}{\mathtt{k}}+\frac{3}{2},1),(\frac{3}{2},1)%
\end{array}%
\bigg|-\frac{cu^{2}}{4\mathtt{k}}\right]
\end{align}
and inverse Sumudu transform of $\mathtt{k}$-Struve function is given by
\begin{align*}
S^{-1}\left[\mathtt{S}_{\nu,c}^{\mathtt{k}}(x)\right]&=S^{-1}\left[\sum_{r=0}^{\infty}\frac{(-c)^{r}}{\Gamma_{\mathtt{k}}\left(r\mathtt{k}+\nu+\frac{3\mathtt{k}}{2}\right)\Gamma\left(r+\frac{3}{2}\right)}\left(\frac{u}{2}\right)^{\frac{\nu}{\mathtt{k}}+1+2r}\right]\\
&=\sum_{r=0}^{\infty}\frac{(-c)^{r}\left(\frac{1}{2}\right)^{\frac{\nu}{\mathtt{k}}+1+2r}}{\Gamma_{\mathtt{k}}\left(r\mathtt{k}+\nu+\frac{3\mathtt{k}}{2}\right)\Gamma\left(r+\frac{3}{2}\right)}
\frac{\left(t\right)^{\frac{\nu}{\mathtt{k}}+2r}}{\Gamma\left(\frac{\nu}{\mathtt{k}}+1+2r\right)}
\end{align*}
Using \eqref{gammak}, we get
\begin{equation}
=\left(\frac{t}{2}\right)^{\frac{\nu}{\mathtt{k}}}k^{-\frac{1}{2}-\frac{\nu}{\mathtt{k}}}\\
{}_{1}\Psi _{3}\left[
\begin{array}{c}
(1,1) \\
(\frac{\nu}{\mathtt{k}}+\frac{3}{2},1),(\frac{3}{2},1), (\frac{\nu}{\mathtt{k}}, 2)%
\end{array}%
\bigg|-\frac{ct^{2}}{4\mathtt{k}}\right]
\end{equation}
In this paper, we consider \eqref{k-Struve} to obtain the solution of the fractional kinetic equations. 
Our methodology herein is based on Sumudu transform,\cite{Belgacem2006a,Belgacem2006b}.  
Fractional calculus is developed to large area of mathematics physics and other engineering applications \cite{Gupta,Gupta,Nisar,Saichev,Saxena1,Saxena2,Saxena3,Saxena4,Zaslavsky,Katatbeh-Belgacem} because of its importance and efficiency.
The fractional differential equation between a chemical reaction or a production scheme (such as in birth-death processes) was established and treated by Haubold and Mathai \cite{Haubold}, (also see \cite{Belgacem2003,Chaurasia,Gupta}). 

\section{Solution of generalized fractional Kinetic equations for $\mathtt{k}$-Struve function}
\label{Sec2}

Let the arbitrary reaction described by a time-dependent quantity $N=\left( N_{t}\right) $. The rate of change $\frac{dN}{dt}$ to be a balance between the destruction rate $\mathfrak{d}$ and the production rate $\mathfrak{p}$ of N, that is, $\frac{%
dN}{dt}=-\mathfrak{d}+\mathfrak{p}$. Generally, destruction and production depend on the quantity N itself, that is, 
\begin{equation}
\frac{dN}{dt}=-\mathfrak{d}\left( N_{t}\right) +\mathfrak{p}\left( N_{t}\right),
\label{eqn-6-Struve}
\end{equation}
where $N_{t}$ described by $N_{t}\left( t^{\ast }\right)
=N\left( t-t^{\ast }\right) ,t^{\ast }>0$.
Another form of ($\ref{eqn-6-Struve}$) is,
\begin{equation}
\frac{dN_{i}}{dt}=-c_{i}N_{i}\left( t\right),  \label{eqn-7-Struve}
\end{equation}
with $N_{i}\left( t=0\right) =N_{0}$, which is the number of density of species $i$ at time $t=0$ and $c_{i}>0$.
The solution of ($\ref{eqn-7-Struve}$) is,
\begin{equation}
N_{i}\left( t\right)=N_{0}e^{-c_{i}t}.  \label{eqn-7a-Struve}
\end{equation}
Integrating ($\ref{eqn-7-Struve}$) gives, 
\begin{equation}
N\left( t\right) -N_{0}=-c.~_{0}D_{t}^{-1}N\left( t\right),
\label{eqn-8-Struve}
\end{equation}
where $_{0}D_{t}^{-1}$ is the special case of the Riemann-Liouville integral
operator and c is a constant. The fractional form of ($\ref{eqn-8-Struve}$) due to \cite%
{Haubold} is,
\begin{equation}
N\left( t\right) -N_{0}=-c_{0}^{\upsilon }D_{t}^{-\upsilon}N\left( t\right),
\label{eqn-9-Struve}
\end{equation}
where $_{0}D_{t}^{-\upsilon }$ defined as 
\begin{equation}
_{0}D_{t}^{-\upsilon }f\left( t\right) =\frac{1}{\Gamma \left( \upsilon
\right) }\int\limits_{0}^{t}\left( t-s\right) ^{\upsilon -1}f\left( s\right)
ds,\Re\left( \upsilon \right) >0.  \label{eqn-9a-Struve}
\end{equation}
Suppose that $f(t)$ is a real or complex valued function of the (time)
variable $t > 0$ and s is a real or complex parameter. The Laplace transform
of $f(t)$ is defined by 
\begin{equation}
F\left( p\right) =L\left[ f(t):p\right] =\int\limits_{0}^{\infty
}e^{-pt}f\left( t\right) dt,\text{ \ \ }\Re\left( p\right) >0
\label{eqn-13-Struve}
\end{equation}
The Mittag-Leffler functions $%
E_\alpha\left(z\right)$ (see \cite{Mittag}) and $E_{\alpha ,\beta }\left( x\right)$ \cite{Wiman} is defined respectively as
\begin{equation}
E_{\alpha}\left( z\right) =\sum_{n=0}^{\infty }\frac{z^{n}}{\Gamma \left(
\alpha n+1 \right) }\text{ \ \ }\left(z,\alpha \in \mathbb{C}; |z|<0,%
\Re\left( \alpha\right)>0\right) .  \label{Mittag1}
\end{equation}
\begin{equation}
E_{\alpha ,\beta }\left( x\right) =\sum_{n=0}^{\infty }\frac{x^{n}}{\Gamma
\left( \alpha n+\beta \right) }.  \label{Mittag2}
\end{equation}

\begin{theorem}
\label{Th1}If $d>0,\nu >0, \mu, c, t\in \mathbb{C}$ and $l>-\frac{3}{2}\mathtt{k}$ then the solution of generalized fractional kinetic equation%
\begin{equation}
N\left( t\right) =N_{0}~\mathtt{S}_{\mu,c}^{\mathtt{k}}\left( d^{\nu}t^{\nu}\right)-d^{\nu }\text{ }%
_{0}D_{t}^{-\nu }N\left( t\right) ,  \label{eqn-14-Struve}
\end{equation}
is given by the following formula%
\begin{equation}
N\left( t\right) =N_{0}\sum_{r=0}^{\infty }\frac{\left( -c\right) ^{r}\Gamma\left[\nu\left(2r+\frac{\mu}{\mathtt{k}}+1\right)+1\right]}
{\Gamma_{\mathtt{k}}\left( r\mathtt{k}+\mu+\frac{3}{2}\mathtt{k}\right) \Gamma \left(r+\frac{3}{2}\right)}
\frac{1}{t}\left(\frac{d^{\nu}t^{\nu}}{2}\right)^{2r+\frac{\mu}{\mathtt{k}}+1}E_{\nu,\nu(2r+\frac{\mu}{\mathtt{k}})+1}\left(
-d^{\nu }t^{\nu }\right) .  \label{eqn-15-Struve}
\end{equation}
where $E_{\nu,\nu(2r+\frac{\mu}{\mathtt{k}})+1}\left(-d^{\nu }t^{\nu }\right)$ is given in \eqref{Mittag2}
\end{theorem}

\begin{proof}
The Sumudu transform of Riemann-Lioville fractional integral operators is given by 
\begin{align}\label{SRL}
S\left\{{}_0D_{t}^{-\nu}f(t);u\right\}=u^{\nu}G(u),
\end{align}
where $G(u)$ is defined in \eqref{Gu}. Now applying Sumudu transform both sides of 
\eqref{eqn-14-Struve} and applying the definition of $\mathtt{k}$-Struve function given in \eqref{k-Struve}, we have
\begin{align*}
N^{*}(u)&=S\left[N\left( t\right);u\right] \\
&=N_{0}S\left[ \mathtt{S}_{\mu,c}^{\mathtt{k}}\left( d^{\nu}t^{\nu}\right) ;u\right] -d^{\nu }S\left[{}_0D_{t}^{-\nu }N\left( t\right) ;u\right]\\
&=N_{0}\left[\int_{0}^{\infty }e^{-pt}\sum_{r=0}^{\infty }\frac{\left( -c\right) ^{r}}{\Gamma_{\mathtt{k}} \left( r\mathtt{k}+\mu+\frac{3}{2}\mathtt{k}\right) \Gamma\left( r+\frac{3}{2}\right) }\left( \frac{d^{\nu}(ut)^{\nu}}{2}\right) ^{2r+\frac{\mu}{\mathtt{k}}+1}dt\right]\\
&-d^{\nu }u^{\nu }N^{*}\left(u\right),
\end{align*}
where 
\begin{align}\label{Sgamma}
S\left\{t^{\mu-1}\right\}=u^{\mu-1}\Gamma(\mu).
\end{align}
By rearranging terms we get,
\begin{align*}
&N^{*}\left(u\right) +d^{\nu}u^{\nu }N^{*}\left( u\right)\\
&=N_{0}\sum_{r=0}^{\infty }\frac{\left( -c\right) ^{r}}{\Gamma_{\mathtt{k}} \left( r\mathtt{k}+\mu+\frac{3}{2}\mathtt{k}\right) \Gamma\left( r+\frac{3}{2}\right) }\left( \frac{d^{\nu}}{2}\right) ^{2r+\frac{\mu}{\mathtt{k}}+1}\\
&\times\int\limits_{0}^{\infty }e^{-t}\left(ut\right)^{\nu(2r+\frac{\mu}{\mathtt{k}}+1)}dt \\
&=N_{0}\sum_{r=0}^{\infty }\frac{\left( -c\right) ^{r}\Gamma[\nu(2r+\frac{\mu}{\mathtt{}k}+1)+1]}{\Gamma_{\mathtt{k}} \left( r\mathtt{k}+\mu+\frac{3}{2}\mathtt{k}\right) \Gamma\left( r+\frac{3}{2}\right) }\left(\frac{u^{\nu}d^{\nu}}{2}\right)^{2r+\frac{\mu}{\mathtt{k}}+1}
\end{align*}
Therefore
\begin{align}\label{th1-pf1}
N^{*}\left(u\right)&=N_{0}\sum_{r=0}^{\infty }\frac{\left( -c\right) ^{r}\Gamma[\nu(2r+\frac{\mu}{\mathtt{}k}+1)+1]}{\Gamma_{\mathtt{k}} \left( r\mathtt{k}+\mu+\frac{3}{2}\mathtt{k}\right) \Gamma\left( r+\frac{3}{2}\right) }\left(\frac{d^{\nu}}{2}\right)^{2r+\frac{\mu}{\mathtt{k}}+1}\notag\\
&\times \left\{u^{\nu(2r+\frac{\mu}{\mathtt{k}}+1)}\sum_{n=0}^{\infty}\left[-(du)^{\nu}\right]^{n}\right\}
\end{align}
Taking inverse Sumudu transform of ($\ref{th1-pf1}$), and by using
\begin{align}\label{pf2}
S^{-1}\left\{ u^{\nu };t\right\} =\frac{t^{\nu -1}}{\Gamma \left(
\nu \right) },\Re\left( \nu \right) >0, 
\end{align}
we have%
\begin{align*}
S^{-1}\left\{N^{*}\left(u\right)\right\}&=N_{0}\sum_{r=0}^{\infty }\frac{\left(-c\right) ^{r}\Gamma[\nu(2r+\frac{\mu}{\mathtt{}k}+1)+1]}{\Gamma_{\mathtt{k}} \left( r\mathtt{k}+\mu+\frac{3}{2}\mathtt{k}\right) \Gamma\left( r+\frac{3}{2}\right) }\left(\frac{d^{\nu}}{2}\right)^{2r+\frac{\mu}{\mathtt{k}}+1}\notag\\
&\times S^{-1}\left\{\sum_{n=0}^{\infty}(-1)^{n}(d)^{\nu n}u^{\nu(2r+\frac{\mu}{\mathtt{k}}+n+1)}\right\}
\end{align*}
which gives,
\begin{align*}
N(t)&=N_{0}\sum_{r=0}^{\infty }\frac{\left(-c\right) ^{r}\Gamma[\nu(2r+\frac{\mu}{\mathtt{k}}+1)+1]}{\Gamma_{\mathtt{k}} \left( r\mathtt{k}+\mu+\frac{3}{2}\mathtt{k}\right) \Gamma\left( r+\frac{3}{2}\right) }\left(\frac{d^{\nu}}{2}\right)^{2r+\frac{\mu}{\mathtt{k}}+1}\\
&\times\left\{\sum_{n=0}^{\infty}(-1)^{n}(d)^{\nu n}\frac{t^{\nu\left(2r+\frac{\mu}{\mathtt{k}}+n+1\right)-1}}{\Gamma\left[\nu\left(2r+\frac{\mu}{\mathtt{k}}+n+1\right)\right]}\right\}\\
&=N_{0}\sum_{r=0}^{\infty }\frac{\left(-c\right)^{r}\Gamma[\nu(2r+\frac{\mu}{\mathtt{k}}+1)+1]}{\Gamma_{\mathtt{k}} \left( r\mathtt{k}+\mu+\frac{3}{2}\mathtt{k}\right) \Gamma\left( r+\frac{3}{2}\right)}\frac{1}{t}\left(\frac{d^{\nu}t^{\nu}}{2}\right)^{2r+\frac{\mu}{\mathtt{k}}+1}\\
&\times\left\{\sum_{n=0}^{\infty}(-1)^{n}(d)^{\nu n}\frac{t^{\nu}}{\Gamma\left[\nu\left(2r+\frac{\mu}{\mathtt{k}}+n+1\right)\right]}\right\}\\
&=N_{0}\sum_{r=0}^{\infty }\frac{\left( -c\right) ^{r}\Gamma\left[\nu\left(2r+\frac{\mu}{\mathtt{k}}+1\right)+1\right]}
{\Gamma_{\mathtt{k}}\left( r\mathtt{k}+\mu+\frac{3}{2}\mathtt{k}\right) \Gamma \left(r+\frac{3}{2}\right)}
\frac{1}{t}\left(\frac{d^{\nu}t^{\nu}}{2}\right)^{2r+\frac{\mu}{\mathtt{k}}+1}\\
&\times E_{\nu,\nu(2r+\frac{\mu}{\mathtt{k}})+1}\left(-d^{\nu }t^{\nu }\right)
\end{align*}
which is the desired result.
\end{proof}
\begin{corollary}\label{Cor-1}
If we put $\mathtt{k}=1$ in $\left( \ref{eqn-15-Struve}\right) $ then we get the solution of fractional kinetic equation involving classical Struve function as:\\
If $d>0,\nu >0, \mu, c, t\in \mathbb{C}$ and $l>-\frac{3}{2}$ then the solution of generalized fractional kinetic equation%
\begin{equation}
N\left( t\right) =N_{0}~\mathtt{S}_{\mu,c}^{1}\left( d^{\nu}t^{\nu}\right)-d^{\nu }\text{ }%
_{0}D_{t}^{-\nu }N\left( t\right) ,  \label{eqn-16-Struve}
\end{equation}
is given by the following formula%
\begin{equation}
N\left( t\right) =N_{0}\sum_{r=0}^{\infty }\frac{\left( -c\right) ^{r}\Gamma\left[\nu\left(2r+\mu+1\right)+1\right]}
{\Gamma\left( r+\mu+\frac{3}{2}\right) \Gamma \left(r+\frac{3}{2}\right)}
\frac{1}{t}\left(\frac{d^{\nu}t^{\nu}}{2}\right)^{2r+\mu+1}E_{\nu,\nu(2r+\mu)+1}\left(
-d^{\nu }t^{\nu }\right) .  \label{eqn-17-Struve}
\end{equation}
\end{corollary}

\begin{theorem}\label{Th2}
If $\mathfrak{a}>0, d>0, \nu >0,c,\mu,t\in \mathbb{C}, \mathfrak{a}\neq d$ and $\mu>-\frac{3}{2}\mathtt{k}$, then the equation %
\begin{equation}
N\left( t\right) =N_{0}~\mathtt{S}_{\mu,c}^{\mathtt{k}}\left( {\mathfrak{a}}^{\nu}t^{\nu}\right)-d^{\nu }\text{ }%
_{0}D_{t}^{-\nu }N\left( t\right) ,  \label{eqn-18-Struve}
\end{equation}
is given by the following formula%
\begin{align}
N\left( t\right) &=N_{0}\sum_{r=0}^{\infty }\frac{\left( -c\right) ^{r}\Gamma\left[\nu\left(2r+\frac{\mu}{\mathtt{k}}+1\right)+1\right]}
{\Gamma_{\mathtt{k}}\left( r\mathtt{k}+\mu+\frac{3}{2}\mathtt{k}\right) \Gamma \left(r+\frac{3}{2}\right)}
\frac{1}{t}\left(\frac{d^{\nu}t^{\nu}}{2}\right)^{2r+\frac{\mu}{\mathtt{k}}+1}\notag\\
&\times E_{\nu,\nu(2r+\frac{\mu}{\mathtt{k}})+1}\left(
-{\mathfrak{a}}^{\nu }t^{\nu }\right) .  \label{eqn-19-Struve}
\end{align}
where $E_{\nu,\nu(2r+\frac{\mu}{\mathtt{k}})+1}(.)$ is given in \eqref{Mittag2}
\end{theorem}

\begin{proof}
Theorem \ref{Th2} can be proved in parallel with the proof of Theorem \ref%
{Th1}. So the details of proofs are omitted.
\end{proof}

\begin{corollary}\label{cor-2}
By putting $\mathtt{k}=1$ in Theorem \ref{Th2}, we get the solution of fractional kinetic equation involving classical Struve function:
If $\mathfrak{a}>0, d>0, \nu >0,c,\mu,t\in \mathbb{C}, \mathfrak{a}\neq d$ and $\mu>-\frac{3}{2}$, then the equation %
\begin{equation}
N\left( t\right) =N_{0}~\mathtt{S}_{\mu,c}^{1}\left( {\mathfrak{a}}^{\nu}t^{\nu}\right)-d^{\nu }\text{ }%
_{0}D_{t}^{-\nu }N\left( t\right) ,  \label{eqn-20-Struve}
\end{equation}
is given by the following formula%
\begin{align}
N\left( t\right) &=N_{0}\sum_{r=0}^{\infty }\frac{\left( -c\right) ^{r}\Gamma\left[\nu\left(2r+\mu+1\right)+1\right]}
{\Gamma\left( r+\mu+\frac{3}{2}\right) \Gamma \left(r+\frac{3}{2}\right)}
\frac{1}{t}\left(\frac{d^{\nu}t^{\nu}}{2}\right)^{2r+\mu+1}\notag\\
&\times E_{\nu,\nu(2r+\mu)+1}\left(
-{\mathfrak{a}}^{\nu }t^{\nu }\right) .  \label{eqn-21-Struve}
\end{align}
\end{corollary}

\begin{theorem}\label{Th3}
If $d>0,\nu >0,c,\mu,t\in \mathbb{C}$
and $\mu>-\frac{3}{2}\mathtt{k}$, then the solution of the following equation%
\begin{equation}
N\left( t\right) =N_{0}~\mathtt{S}_{\mu,c}^{\mathtt{k}}\left(t^{\nu}\right)-d^{\nu}
_{0}D_{t}^{-\nu }N\left( t\right) ,  \label{eqn-22-Struve}
\end{equation}
is given by the following formula%
\begin{align}
N\left( t\right) &=N_{0}\sum_{r=0}^{\infty }\frac{\left( -c\right) ^{r}\Gamma\left[\nu\left(2r+\frac{\mu}{\mathtt{k}}+1\right)+1\right]}
{\Gamma_{\mathtt{k}}\left( r\mathtt{k}+\mu+\frac{3}{2}\mathtt{k}\right) \Gamma \left(r+\frac{3}{2}\right)}
\frac{1}{t}\left(\frac{t}{2}\right)^{2r+\frac{\mu}{\mathtt{k}}+1}\notag\\
&\times E_{\nu,\nu(2r+\frac{\mu}{\mathtt{k}})+1}\left(
-{{d}}^{\nu }t^{\nu }\right) .  \label{eqn-23-Struve}
\end{align}
where $E_{\nu,\nu(2r+\frac{\mu}{\mathtt{k}})+1}(.)$ is given in \eqref{Mittag2}
\end{theorem}

\begin{corollary}
\label{Cor3}
If we set $\mathtt{k}=1$ then \eqref{eqn-23-Struve} reduced as follows:\\
If $d>0,\nu >0,c,\mu,t\in \mathbb{C}$
and $\mu>-\frac{3}{2}$, then the solution of the following equation%
\begin{equation}
N\left( t\right) =N_{0}~\mathtt{S}_{\mu,c}^{1}\left(t^{\nu}\right)-d^{\nu}
_{0}D_{t}^{-\nu }N\left( t\right) ,  \label{eqn-24-Struve}
\end{equation}
is given by the following formula%
\begin{align}
N\left( t\right) &=N_{0}\sum_{r=0}^{\infty }\frac{\left( -c\right) ^{r}\Gamma\left[\nu\left(2r+\mu+1\right)+1\right]}
{\Gamma\left( r+\mu+\frac{3}{2}\right) \Gamma \left(r+\frac{3}{2}\right)}
\frac{1}{t}\left(\frac{t}{2}\right)^{2r+\mu+1}\notag\\
&\times E_{\nu,\nu(2r+\mu)+1}\left(
-{{d}}^{\nu }t^{\nu }\right) .  \label{eqn-25-Struve}
\end{align}
\end{corollary}

\section{Graphical interpretation}
In this section we plot the graphs of our solutions of the fractional kinetic equation, 
which is established in \eqref{eqn-15-Struve}. In each graph, we gave three solutions of 
the results on the basis of assigning different values to the parameters.In figures 1, 
we take  $\mathtt{k}=1$ and $\nu=0.5, 0.7, 0.9, 1, 1.5$. Similarly figures 2-3 are 
plotted respectively by taking $\mathtt{k}=2,3$. Figures 4-6 are plotted by considering 
the solution given in \eqref{eqn-23-Struve} by taking $\nu=0.5, 0.7, 0.9, 1, 1.5$ 
and $\mathtt{k}=1,2,3$. Other than $\nu$ and $\mathtt{k}$ all other parameters are fixed 
by 1. It is clear from these figures that $N_{t}>0$ for $t > 0$ and $N_{t}>0$ is monotonic  
increasing function for $t\in (0,\infty)$. In this study, we choose first 50 terms of 
Mittag-Leffler function and first 50 terms of our solutions to plot the graphs. 
$N_{t}=0$ , when $t>0$ and $N_{t}\rightarrow\infty$ when $t\rightarrow1$ for all 
values of the parameters.

%
%
%
%
%

\end{document}